\documentclass[11pt,epsfig]{article}

\usepackage{fullpage}
\usepackage{color}
\usepackage{amsmath}
\usepackage{amsfonts}
\usepackage{amssymb}
\usepackage{amsthm}
\usepackage{newlfont}
\usepackage{graphicx}
\newtheorem{thm}{Theorem}[section]
\newtheorem{prop}{Proposition}[section]
\newtheorem{lem}{Lemma}[section]
\theoremstyle{remark}
\newtheorem{rem}{Remark}[section]
\newtheorem*{rem*}{Remark}
\newtheorem{cor}{Corollary}[section]
\numberwithin{equation}{section}

\newcommand{\ed}{\end {document}}

\title{Stability of the semi-implicit method for the Cahn-Hilliard equation with logarithmic potentials}
\author
{Dong Li
\thanks
{Department of Mathematics, the Hong Kong University of Science \& Technology,
Clear Water Bay, Kowloon, Hong Kong. Email: {madli@ust.hk}.
The author's work was supported in part by Hong Kong RGC grant GRF 16307317
 and 16309518. }\qquad
{Tao Tang}
\thanks{   Division of Science and Technology, BNU-HKBU United International College,
    Zhuhai, Guangdong Province, China; and SUSTech International Center for Mathematics, 
    Southern University of Science and Technology, Shenzhen, China.
    Email: tangt@sustech.edu.cn. This author's work is partially supported by the NSFC
    grants 11731006 and K20911001,  NSFC/RGC 11961160718, and the Science Challenge Project (No. TZ2018001).}
}

\begin{document}
\date{}
\maketitle
\begin{abstract}
We consider the two-dimensional Cahn-Hilliard equation with  logarithmic potentials and
periodic boundary conditions. We employ the standard semi-implicit numerical  scheme which treats
 the linear fourth-order dissipation term implicitly and the nonlinear term explicitly. Under natural
 constraints on the time step we prove strict phase separation and energy stability of the semi-implicit scheme. This appears to be the first rigorous result for the semi-implicit discretization
 of the Cahn-Hilliard equation with singular potentials. 
\end{abstract}

\section{Introduction}

Consider the 2D Cahn-Hilliard equation on $\Omega=\mathbb T^2=[-\pi,\pi)^2$:
\begin{align} \label{1}
\begin{cases}
\partial_t u = \Delta \mu=\Delta ( - \nu \Delta u  + F^{\prime}(u) ), 
\qquad (t, x) \in (0,\infty) \times \Omega; \\
u\bigr|_{t=0}=u_0,
\end{cases}
\end{align}
where $u: \Omega\to (-1,1)$ is the order parameter of a two-phase system such
as a binary alloy, and the term $\mu$ denotes the chemical potential.
The two end-points
$u=\pm 1$ correspond to pure states.  The coefficient $\nu>0$ denotes mobility. In this paper
we take it  to be a constant parameter. 
The thermodynamic potential $F:\, (-1,1 )
\to \mathbb R$ is given by 
\begin{align}
&F(u) = \frac {\theta} 2 \Bigl(
(1+u)\ln (1+u) + (1-u) \ln (1-u) \Bigr) -\frac{\theta_c} 2 u^2,  \quad 0<\theta<\theta_c;
\label{1.2a}\\
&f(u)=F^{\prime}(u)= -\theta_c u + \frac {\theta}2 \ln \frac{1+u}{1-u}
=:-\theta_c u +\tilde f(u), \qquad
F^{\prime\prime}(u) = \frac {\theta}{1-u^2} - \theta_c, \label{1.2b}
\end{align}
where the logarithmic part accounts for the entropy of mixing. The parameters $\theta$
and $\theta_c$ corresponds to the absolute temperature and the
 critical temperature respectively.  Denote by $u_+>0$ the positive root of the equation 
$f(u)=0$ (see \eqref{1.2b}). 
Under the condition $0<\theta<\theta_c$ 
the potential $F$ takes the form of  a double-well with two equal minima
at $u_+$ and $-u_+$ which are usually called binodal points. One should note that the 
condition $0<\theta<\theta_c$ is of physical importance since it guarantees that
that $F$ has a double-well form and  phase separation can indeed occur. 
For $u_s= (1-{\theta}/{\theta_c} )^{\frac 12}$, the region $(-u_s, u_s)$   where $F^{\prime\prime}(u)<0$ is called the spinodal interval. If the quenching is shallow, i.e., the temperature $\theta$ is close to
the absolute temperature $\theta_c$, then one can expand near $u=0$ and obtain
the usual quartic polynomial approximation of the free energy. 

The usual energy conservation takes the form:
\begin{align} \label{1.4}
\frac d {dt } \mathcal E (u ) = - \| |\nabla|^{-1} \partial_t u \|_2^2,
\qquad \mathcal E (u ) = \int_{\Omega}
(\frac 12 \nu |\nabla u|^2 +F(u) ) dx. 
\end{align}
Note that for $u\in (-1,1)$, the term $F(u)$ is bounded by an absolute constant, and the
only coercive quantity in $\mathcal E(u)$ is the gradient term. 
\begin{rem}
We note that the usual quartic polynomial approximation of the free energy 
$F(u)$ is given by (below the series converges for $u\in [-1,1]$)
\begin{align*}
F(u) &= - \frac {\theta_c} 2 u^2 + \theta \sum_{k=0}^{\infty}
\frac {u^{2k+2} } {(2k+1)(2k+2)} \notag \\
&
\approx F_{\operatorname{quartic}}(u)= \frac{\theta}2 \cdot \frac {u^4}6 
+(\frac{\theta}2  -\frac {\theta_c}2 )u^2.
\end{align*}
The standard double-well
potential $\operatorname{const} \cdot (u^2-1)^2$ corresponds to the specific 
choice $\theta/\theta_c=3/4$.
However, this approximation introduces  a nontrivial 
shift of the location of the minimum. Namely
for the original free energy $F(u)$, its two equal minima occur at $\pm u_+$, where 
 $u_+>0$ is the positive root of the equation $f(u)=0$ (see \eqref{1.2b}). 
 In particular, $0<u_+<1$. In contrast, the standard double well
potential $F_{\operatorname{quartic, standard}}= (u^2-1)^2/4$ has minima at
$u= \pm 1$. We should point out that, in view of the two minima $\pm u_+$ which are well
inside the region $(-1,1)$ and singularity of the derivative of the potential, 
it is in some sense natural to expect strict phase separation for the evolution
equation, i.e., 
\begin{align}
\|u\|_{\infty}\le 1-\delta_0<1, \qquad  \text{for some $\delta_0>0$}.
\end{align} 
The strict phase separation turns out to play an important role in the rigorous analysis
of \eqref{1}.
\end{rem}

Mathematically speaking, the system \eqref{1} can be recast as
 a gradient flow of a Ginzburg-Landau (GL) type energy
functional $\psi(u) $ in $H^{-1}$, i.e., 
\begin{align} \label{1.6}
& \partial_t u= - \frac {\delta \psi} {\delta u } \Bigr|_{H^{-1}} =
\Delta \left( \frac {\delta \psi} {\delta u}\right|_{L^2} \biggr), 
\end{align}
where $\frac {\delta \psi}{\delta u} \Bigr|_{H^{-1}}$ , $\frac {\delta \psi}
{\delta u}\Bigr|_{L^2}$ denote the standard variational derivatives in $H^{-1}$ and $L^2$ respectively, and
\begin{align} \label{1.7}
&\psi(u)= \int_{\Omega} ( \frac 12 \nu |\nabla u|^2 + F(u) ) dx.
\end{align}
Here the gradient term in the GL energy accounts for surface tension effects, or more generally, 
short range interactions in the material. This particular form of energy functional  can be
derived from
an approximation of a nonlocal term representing long range interactions \cite{CH58}. 
In \cite{GL97, GL98} Giacomin and Lebowitz considered a lattice gas model with certain
long range Kac potentials, and gave a 
 rigorous derivation of the nonlocal Cahn-Hilliard equation.
 Further results such as regularity and traveling waves on these and similar
models can be found in \cite{BH05, GZ03, Wang97} and the references therein. 

For the Cahn-Hilliard equation with constant mobility and logarithmic
potentials,  Elliott and Luckhaus in \cite{EL91} 
considered the case of a multi-component mixture in a finite domain with
 Neumann boundary conditions and proved  that if the initial data $u_0 \in H^1$ satisfies $\|u_0\|_{\infty}\le 1$ with space average in $(-1,1)$, then
there exists a unique global solution $u \in C_t^0 H^{-1} \cap L_t^{\infty} H^1_x$,
$\partial_t u \in L^2_{t,\operatorname{loc}} H^{-1}_x$, $\sqrt t \partial_t u \in L_{t,\operatorname{loc}}^2 H^1_x$  and $\|u \|_{\infty} \le 1$. Furthermore, it was shown that
the set $\{|u|=1\}$ has measure zero so that there are no singularities in the potential. 
The key idea in \cite{EL91} is to use regularization and replace
 the logarithmic term by a smoothed version:
\begin{align}
\phi_{\epsilon} (r) = \begin{cases}
\ln r, \qquad r \ge \epsilon;\\
\ln \epsilon -1 + \frac {r} {\epsilon}, \quad r<\epsilon.
\end{cases}
\end{align}
The main point is to derive $\epsilon$-independent estimates on the regularized problem
and extract the desired solution in the vanishing $\epsilon$-limit. 
In \cite{DD95}, Debussche and Dettori adopted a different regularization of $F(u)$:
\begin{align}
F_N(u) = - \frac {\theta_c} 2 u^2 + \theta \sum_{k=0}^{N}
\frac {u^{2k+2} } {(2k+1)(2k+2)}. 
\end{align}
For $L^2$ or $H^1$ initial data $u_0$ with $\|u_0\|_{\infty} \le 1$, $m(u_0)\in (-1,1)$ 
with either Neumann or periodic boundary conditions, they proved the
existence and uniqueness of solutions as well as  continuity of the semigroup. 
In \cite{MZ04} Miranwille and Zelik introduced another novel approximation by using the viscous  Cahn-Hilliard
equations, namely 
\begin{align}
\begin{cases}
\epsilon \partial_t u + (-\Delta)_N^{-1} \partial_t u 
=\Delta u - f(u) + \langle f(u) \rangle, \quad \epsilon >0;\\
\partial_n u\Bigr|_{\partial \Omega}=0,
\end{cases}
\end{align}
where $\langle v \rangle := |\Omega|^{-1} \int_{\Omega} v(x) dx$ and
$(-\Delta)_N^{-1}$ denotes the inverse Laplacian with Neumann boundary
conditions acting on the space $L_0^2(\Omega)= \{v \in L^2(\Omega):\, \langle v \rangle =0 \}$. 
In \cite{AW07} Abels and Wilke  employed a different approach based on the powerful theory of monotone operators. It is worthwhile
pointing out that, to show the subgradient $\partial F(c)$ is single-valued (see Theorem 4.3
on P3183 of \cite{AW07} and the proof therein), one needs some suitable approximation of the potential by smooth ones 
(since the derivative goes to $\pm \infty$ at the end-points) and carefully derive the limits. In a related work \cite{Ken95}, Kenmochi,
Niezg\'odka and Pawlow studied a very general version of Cahn-Hilliard equation involving a multivalued mapping by using  sub-differential operator theory. 
 The approach therein is  based on
several approximation procedures using smoothed equations and potentials. We note that
more recently there has been some new developments on the analysis of the Cahn-Hilliard
equation with singular potentials and dynamical boundary conditions \cite{Wu1, Wu2}. 
Regarding the prior state of the art literature on these topics and more classical theory concerning
 long time behavior and attractors,  we refer the interested readers to
  \cite{CMZ11, Wu1, Wu2, Mir17} and the references therein for more in-depth reviews and
discussions.

There are some subtle technical difficulties associated with the numerical
discretization of \eqref{1}. We now point out two most pronounced issues.
Denote $u^n \approx u(t_n)$ as the numerical solution at time step $t_n= n\tau$, where
$\tau>0$ is the time step.

\begin{enumerate}
\item How do we guarantee that $u^n \in (-1, 1)$ for all $n\ge 0$?

\item How to ensure the energy decay property: $\mathcal E (u^{n+1}) \le \mathcal E(u^n)$
for all $n\ge 0$?
\end{enumerate}

One should note that the first issue is already present for the continuous PDE solutions. As 
was already mentioned earlier,
the traditional route to solving this problem is via regularization of the nonlinearity or using
the technique of sub-differential operators. The regularization technique can be transferred
and modeled  on the numerical discretization especially for the existence of solutions (for 
 implicit schemes).  Indeed in \cite{EC92} developing upon the earlier work
\cite{EL91}, 
Copetti and Elliott \cite{EC92} considered a fully  implicit Euler scheme applied to the Cahn-Hilliard equation with a finite element approximation in space. It was shown 
that under the condition that the time step $\tau$ is sufficiently
small, and if the initial data satisfies $\|u_0\|_{\infty} \le 1$,   
$\overline{u_0} <1-\delta<1$,  then there exists a unique numerical solution for the implicit Euler discretization, 
satisfying $\|u^n\|_{\infty} <1$ for all $n\ge 1$. In \cite{BEC96} the authors generalized the
approach in \cite{EC92} to the multi-component Cahn-Hilliard flow. It should be noted that,
due to the implicit treatment of the (concave) diffusion term,  the energy stability and time
step constraint is not unconditional in \cite{EC92}. This can be rectified using D. Eyre's 
convex-splitting technique which is recently adopted in \cite{Wang19} using centered difference discretization in space.  We note that the convex-splitting technique belongs to so-called partially
implicit methods \cite{Tang20} for which the convex part of the nonlinearity is treated implicitly. By using a variational idea taking advantage of the singular nature of the nonlinearity, the numerical solution constructed
in \cite{Wang19} can be guaranteed to lie in the interval $[-1,1]$ in each iteration. 
However, for semi-implicit methods, this line of argument completely breaks down, and
to our best knowledge, this issue was completely open.

Whilst the first issue already presents itself a fundamental problem for semi-implicit methods, the second
one is even more serious.  As it turns out the explicit or implicit treatment of the nonlinear term 
can lead to a fundamental change of the energy stability of the associated iterative system. 
The analysis of energy stability gives a clear picture
 why implicit methods (or
partially implicit methods) are usually favored/adopted in the literature (see also recent \cite{Yang17}). To elucidate the discussion we
shall compare the usual semi-implicit methods with the implicit methods in the next two subsections.
For simplicity we assume the ideal
scenario that all $u^n \in (-1,1)$. 

\subsubsection*{The usual semi-implicit discretization case.}
 A typical semi-implicit discretization takes the form
\begin{align*}
\frac {u^{n+1} - u^n }{\tau}
= - \nu \Delta^2 u^{n+1}  +\Delta (f (u^n) ). 
\end{align*}
Multiplying both sides by $(-\Delta)^{-1} (u^{n+1} -u^n)$ (note that $\overline{u^{n+1}}
=\overline{u^n}$, and see  \eqref{1.23} for the definition of $(-\Delta)^{-1}$ and
$\overline u$) and integrating by parts, we obtain
\begin{align} \label{b1_e1}
\frac 1 {\tau} &\| |\nabla|^{-1} (u^{n+1} -u^n) \|_2^2 
+\frac {\nu}2 \| \nabla (u^{n+1} -u^n ) \|_2^2  +\mathcal E (u^{n+1})
-\mathcal E (u^n) = \int_{\Omega} H_1 dx, 
\end{align}
where $|\nabla|^{-1}=(-\Delta)^{-\frac 12}$ (see \eqref{1.22}), and 
$
H_1 = {F(u^{n+1}) - F(u^n) - f(u^n) (u^{n+1}-u^n)}. 
$

Now recall that $F^{\prime\prime}(\xi ) = \frac {\theta}{1-\xi^2} - \theta_c$. Clearly then
\begin{align}
H_1= \frac 12\left(\frac  {\theta}{1- \xi_0^2}   -\theta_c\right) \cdot (u^{n+1}-u^n)^2, 
\end{align}
where $\xi_0$ is a function with values sandwiched between $u^n$ and $u^{n+1}$.  
Note that on the LHS of \eqref{b1_e1},  we have the usual estimate
\begin{align}
\frac 1 {\tau} &\| |\nabla|^{-1} (u^{n+1} -u^n) \|_2^2 
+\frac {\nu}2 \| \nabla (u^{n+1} -u^n ) \|_2^2 
\ge  \sqrt{\frac {2\nu}{\tau}} \| u^{n+1}-u^n \|_2^2.
\end{align}
However even for very small $\tau>0$, this is in-sufficient to control the singular pre-factor
$\frac 1 {1-\xi_0^2}$ in the $H_1$-term which could potentially become rather large
when $\xi_0 \to \pm 1$. 

\subsubsection*{The usual implicit discretization case.}
 A typical implicit discretization takes the form
\begin{align}
\frac {u^{n+1} - u^n }{\tau}
= - \nu \Delta^2 u^{n+1}  +\Delta (f (u^{n+1} ) ). 
\end{align}
Multiplying both sides by $(-\Delta)^{-1} (u^{n+1} -u^n)$ and integrating, we obtain
\begin{align} \label{b1_e2}
\frac 1 {\tau} &\| |\nabla|^{-1} (u^{n+1} -u^n) \|_2^2 
+\frac {\nu}2 \| \nabla (u^{n+1} -u^n ) \|_2^2  +\mathcal E (u^{n+1})
-\mathcal E (u^n) = \int_{\Omega} H_2 dx, 
\end{align}
where
$
H_2 = {F(u^{n+1}) - F(u^n) - f(u^{n+1} ) (u^{n+1}-u^n)}. 
$

Now  recall again that $F^{\prime\prime}(\xi ) = \frac {\theta}{1-\xi^2} - \theta_c$. Clearly then
\begin{align}
H_2 & = -F(u^n) +F(u^{n+1}) + f(u^{n+1}) (u^n-u^{n+1}) \notag \\
&=\frac 12 \left(-\frac  {\theta}{1- \xi_0^2} +\theta_c\right)\cdot (u^{n}-u^{n+1} )^2, 
\end{align}
where $\xi_0$ is a function with values sandwiched between $u^n$ and $u^{n+1}$.  
Note that on the LHS of \eqref{b1_e2},  we have the usual estimate
\begin{align}
&\frac 1 {\tau} \| |\nabla|^{-1} (u^{n+1} -u^n) \|_2^2 
+\frac {\nu}2 \| \nabla (u^{n+1} -u^n ) \|_2^2  \notag \\
\ge&\;  \sqrt{\frac {2\nu}{\tau}} \| u^{n+1}-u^n \|_2^2 
\;\ge \;  \frac 12\theta_c \|u^{n+1}-u^n\|_2^2,
\end{align}
if $0<\tau\le  {8\nu}/{\theta_c^2} $.  On the other hand, note that the singular  term
$-\frac 1 {1-\xi_0^2}$ is always negative (provided we guarantee that $u^n$ and 
$u^{n+1}$ always stay inside the interval $(-1,1)$). Thus the energy decay property
can be expected for implicit methods.

Now, from the above comparative discussion in the preceding two subsections, it is clear that there are nontrivial technical 
obstacles for the  semi-implicit methods applied on the Cahn-Hilliard equation with
logarithmic potentials.  Nevertheless, the purpose of this work is to introduce a new framework
to settle these open issues.

Consider the following semi-implicit discretization of \eqref{1}:
\begin{align} \label{Au4e1}
\begin{cases}
\frac {u^{n+1}-u^n} {\tau}
=-\nu \Delta^2 u^{n+1} -\theta_c \Delta u^{n+1} +  \Delta ( \tilde f (u^n) ),
\qquad n\ge 0;\\
u^0=u_0,
\end{cases}
\end{align}
where  $\tilde f(u) = \frac{\theta}2 \ln(\frac {1+u}{1-u} )$. The relation of $\tilde f(u)$
with $f(u)$ is given by \eqref{1.2b}.

\begin{thm}[Stability of the semi-implicit discretization scheme] \label{t1}
Consider the implicit-explicit scheme \eqref{Au4e1} for the phase field equation \eqref{1} with
the logarithmic potential \eqref{1.2a}. 
Assume the initial data $u_0 \in H^5(\mathbb T^2)$ and $\|u_0\|_{\infty} \le 1-\delta_0$ for
some $\delta_0 \in (0,1)$. There exists $\tau_0=\tau_0(\|u_0\|_{H^5}, \delta_0, \nu,
\theta, \theta_c)>0$
such that for any $0<\tau\le \tau_0$, the following hold for \eqref{Au4e1}:

\begin{enumerate}
\item \underline{Unique solvability and propagation of Sobolev regularity}. The iterates $u^n$ are well-defined for all
$n\ge 1$. Furthermore  $\sup_{n\ge 1} \|u^n\|_{H^5(\mathbb T^2)} \le A_1<\infty$ for some
$A_1$ depending only on $(\|u_0\|_{H^5}, \delta_0, \nu, \theta, \theta_c)$. 
\item \underline{Strict phase separation}. There exists a constant $\delta_1 \in (0,1)$ depending only
on $(\|u_0\|_{H^5}, \delta_0, \nu, \theta,\theta_c)$,  such that
$\sup_{n\ge 1} \|u^n\|_{\infty} \le 1-\delta_1$. 

\item \underline{Energy stability}.  $\mathcal E (u^{n+1}) \le \mathcal E(u^n)$ for all $n\ge 0$. 

\end{enumerate}

\end{thm}

To prove Theorem \ref{t1}, we introduce a new strategy which concurrently establishes
the strict phase separation and uniform Sobolev regularity of the iterates $u^n$ through
an inductive procedure. Besides using the discrete energy inequality to control $H^1$-norm
of $u^n$, we employ several bootstrapping long time estimates on the discrete chemical potential $K^n =-\nu \Delta u^n -\theta_c u^n +\tilde f(u^n)$ to gain uniform-in-time higher Sobolev
bounds. This part of the argument is technical and we have to appeal to a delicate dichotomy argument to eliminate some sporadic drift of higher norms of $K^n$ (see Subsection \ref{subs:Knlong} for more
details).  The strict phase separation property of $u^n$ can be deduced through a uniform estimate
on the quantity $g^n = \tilde f(u^n) = \frac{\theta}2 \ln (\frac {1+u^n}{1-u^n})$, which in turn
is obtained by analyzing a nonlinear elliptic problem connecting $g^n$ to $K^n$.  A subtle point in
the whole analysis is to obtain uniform in time estimates which are largely independent of
the induction hypothesis.  In order not to overburden the reader with notations and keep
the analysis relatively simple, we do not optimize the regularity assumption on initial data, and
we do not spell out the precise dependence of the time step constraint on various parameters.
All these issues and further generalizations will be addressed in forthcoming works. 

\begin{rem}
We stress again that the assumption $\|u_0\|_{\infty} \le 1-\delta_0$ is quite natural
from the point of view that the free energy has two equal minima well inside the interval
$(-1,1)$ and its derivative blows up as $u\to \pm 1$. 
\end{rem}

\begin{rem}
A variant of the scheme \eqref{Au4e1} is:
\begin{align} \label{Au6e1}
\frac{u^{n+1}-u^n}{\tau} = -\nu \Delta^2 u^{n+1} +\Delta (f(u^n)),
\end{align}
where $f(u)= -\theta_c u+\frac{\theta}2 \ln(\frac {1+u}{1-u} )$. 
Theorem \ref{t1} also holds for this case. Compared with \eqref{Au4e1}, a slight
difference is the solvability of $u^{n+1}$ in the numerical scheme. In the former case the time step has to be taken suitably
small so that $u^{n+1}$ can be uniquely solved from $u^n$. In the latter case (i.e. \eqref{Au6e1})  the
solvability is not an issue and one can uniquely solve $u^{n+1}$ for any $\tau>0$.
\end{rem}

\begin{rem}
From a more practical point view, one should consider the spectral Galerkin truncated system:
\begin{align} 
\begin{cases}
\frac {u^{n+1}-u^n} {\tau}
=-\nu \Delta^2 u^{n+1} -\theta_c \Delta u^{n+1} +  \Delta \Pi_N ( \tilde f (u^n) ),
\qquad n\ge 0;\\
u^0=\Pi_N u_0,
\end{cases}
\end{align}
where $\Pi_N$ is the projection into first $N$ Fourier modes.  With minor modifications
our analysis can be extended to  this
case. Note that in this case for the phase separation property to hold, we need to
impose it on $u^0=\Pi_N u_0$ since $\Pi_N$ is not a continuous operator in $L^{\infty}$. 
Alternatively by using the high regularity of $u_0$, one can show that
$\lim_{N\to \infty} \| u^0 -\Pi_N u_0\|_{\infty} = \lim_{N\to \infty}
\|\Pi_{>N } u_0 \|_{\infty} =0$.

\end{rem}

As an immediate application of Theorem \ref{t1} (and to make this paper self-contained), we obtain the following 
wellposedness result for the continuous PDE solution to \eqref{1}. As a matter of fact this approach
can be refined to yield a new wellposedness and regularity theory for the continuous case which we
will address elsewhere. For simplicity we do not lower the regularity assumption on the initial data.

\begin{cor} [Existence and uniqueness of  the PDE solution] \label{c1}
Assume the initial data $u_0 \in H^s(\mathbb T^2)$, $s\ge 5$ and $\|u_0\|_{\infty} \le 1-\delta_0$ for
some $\delta_0 \in (0,1)$. Then corresponding to $u_0$ there exists a unique global
solution $u\in C_t^0 H_x^s \cap C_t^1 H_x^{s-4}$ to \eqref{1} satisfying $\sup_{0\le t<\infty}
\|u(t) \|_{H^s} <\infty$ and $\sup_{0\le t <\infty} \|u(t) \|_{\infty} \le 1-\delta_1$
for some $\delta_1 \in (0,1)$.

\end{cor}

Our final result is the error analysis for the semi-implicit scheme. A similar result also holds
for the variant \eqref{Au6e1}.

\begin{thm}[Error analysis] \label{t2}
Let $\nu>0$.  Assume the initial data $u_0 \in H^5(\mathbb T^2)$ and $\|u_0\|_{\infty} \le 1-\delta_0$ for some $\delta_0 \in (0,1)$.  Let $u^n$ be the corresponding numerical
solution constructed in Theorem \ref{t1}. Let $u(t)$ the exact PDE solution to \eqref{1}
constructed in Corollary \ref{c1}.  Let $0<\tau\le \min\{\tau_0,  \frac{\nu}{4\theta_c^2}\}$ 
where $\tau_0$ is the same as in Theorem \ref{t1}. Define $t_m= m \tau$, $m\ge 1$. Then
\begin{align} \label{1.21}
\| u(t_m) - u^m \|_2
\le C_1 e^{C_2 t_m} \tau.
\end{align}
Here $C_1, C_2>0$ depends  on $(u_0,\delta_0, \nu, \theta_c, \theta)$.

\end{thm}

The rest of this paper is organized as follows. In Section 2 we give the proof of Theorem
\ref{t1}. Section 3 is devoted to the proof of Corollary \ref{c1}. In Section 4 we complete
the error analysis and give the proof of Theorem \ref{t2}.  In Section 5 we give some
concluding remarks.

\subsubsection*{Notation}   \label{sec:1.0.1}
For any real number $a\in \mathbb R$, we denote by $a+$ the quantity $a+\epsilon$ for sufficiently small
$\epsilon>0$. The numerical value of $\epsilon$ is unimportant, and the needed
smallness of $\epsilon$ is usually clear from the context. 
The notation $a-$ is similarly defined.  This notation is particularly handy for interpolation inequalities.
For example we shall use the notation
\begin{align*}
\|f \|_{\infty-} 
\end{align*}
to denote $\|f \|_p$ for all large $p<\infty$.

  For any two quantities $X$ and $Y$, we denote $X \lesssim Y$ or $X=O(Y)$ if
$X \le C Y$ for some constant $C>0$. Similarly $X \gtrsim Y$ if $X
\ge CY$ for some $C>0$. We denote $X \sim Y$ if $X\lesssim Y$ and $Y
\lesssim X$. The dependence of the constant $C$ on
other parameters or constants is usually clear from the context and
we will often suppress  this dependence. We denote $X \lesssim_{Z_1,\cdots,Z_m} Y$ if
$X\le C Y$, where the constant $C$ depends on the parameters $Z_1,\cdots,Z_m$.
For any quantities $X_1$, $X_2$, $\cdots$, $X_N$, we denote by $C(X_1,\cdots,
X_N)$ or $C_{X_1,\cdots, X_N}$ a positive constant depending on $(X_1,\cdots, X_N)$.

We denote  by $\mathbb T^2=[-\pi, \pi)^2$ the usual periodic torus in two dimensions.
For a function $f:\mathbb T^2 \to \mathbb R$, we denote by $$\overline f
= \frac 1 {(2\pi)^2} \int_{\mathbb T^2} f(x) dx$$ the average/mean value of $f$ on
$\mathbb T^2$.  We adopt the following convention for the usual Fourier transform
on $\mathbb T^2$ (below assume $f\in C^{\infty}$ for simplicity):
\begin{align*}
&(\mathcal F f)(k)=\widehat{f}(k) = (2\pi)^{-2}\int_{\mathbb T^2} f(x) e^{-i x \cdot k } dx, \quad k \in \mathbb Z^2;\\
&f(x) = \sum_{k \in \mathbb Z^2} \widehat f(k) e^{i k \cdot x}.
\end{align*}
We denote by $|\nabla|^s=(-\Delta)^{\frac s2}$ the operator corresponding to the symbol
$|k|^s$ such that 
\begin{align} \label{1.22}
\widehat{|\nabla|^s f}(k)=|k|^s \widehat f(k).
\end{align}
 Note that for $s<0$, $|\nabla|^s f$ is only defined for smooth functions
$f$ with $\widehat f(0)=0$. For example, if  $f \in L^1(\mathbb T^2)$ and $\overline{f}=0$
(thus $\widehat f(0)=0$), we can
define
\begin{align} \label{1.23}
\mathcal F  ( (-\Delta)^{-1}  f )(k)=  \frac 1 {|k|^2}  \widehat f(k),
\qquad\forall\, 0\ne k \in \mathbb Z^2.
\end{align}
In yet other words, $(-\Delta)^{-1}$ corresponds to the Fourier multiplier $1/|k|^2$ acting
on $L^1$ functions whose $\operatorname{zero}^{\operatorname{th}}$ mode is zero.

For $f$, $g\in L^2(\mathbb T^2\to \mathbb R)$, we denote by $\langle\,, \,\rangle $ the usual
$L^2$-pairing:
\begin{align*}
\langle f, g \rangle = \int_{\mathbb T^2} f(x) g(x) dx.
\end{align*}

\section{Proof of Theorem \ref{t1} }
For simplicity we assume $\nu=1$ in \eqref{1}. 
Let us consider the following semi-implicit scheme:
\begin{align} \label{A1e1}
\frac {u^{n+1}-u^n} {\tau}
=-\Delta^2 u^{n+1} -\theta_c \Delta u^{n+1} +  \Delta ( \tilde f (u^n) ),
\end{align}
where $\tilde f(u) = \frac{\theta}2 \ln(\frac {1+u}{1-u} )$. Then
\begin{align}
u^{n+1} =  \frac 1 {1+\tau \Delta^2+\tau \theta_c \Delta} u^n + \frac {\tau \Delta}{1+\tau \Delta^2
+\tau \theta_c \Delta}
(\tilde f(u^n) ).
\end{align}
Note that for $0\ne k \in \mathbb Z^2$, $1+\tau |k|^4 \ge 2 \sqrt{\tau} |k|^2 \ge
\theta_c \tau |k|^2$ if $0<\tau\le \frac 4 {\theta_c^2}$.   We shall assume the slightly stronger condition $0<\tau\le \frac 2{\theta_c^2}$ to ensure solvability.

For convenience we shall sometimes denote
\begin{align*}
g^n = \tilde f( u^n) = \frac {\theta}2 \ln \frac{1+u^n}{1-u^n}.
\end{align*}

The inductive assumption is: 
\begin{align*}
&\| g^n \|_{H^2} \le A_0<\infty; \quad
\| u^n\|_{H^5} \le A_1<\infty.
\end{align*}
The choice of the constants $A_0$ and $A_1$ will become clear in the course of the proof.
The base step $n=0$ clearly holds true. In the rest of the proof we shall focus on the induction
step $n\Rightarrow n+1$ for general $n$.

From the estimate of $g^n$,  it follows that $\|u^n\|_{\infty} \le 1-\delta_1<1$ for some $\delta_1>0$.   Also clearly by using the iterative relation, 
\begin{align*}
\|u^{n+1} \|_{H^5} \le C_{A_1} <\infty.
\end{align*}

Thus $\| P_{>N} (u^{n+1} -u^n) \|_{\infty} \le \frac {\delta_1}{4}$ if $N$ is sufficiently large (here 
$P_{>N}$ is the usual Littlewood-Paley projector adapted to frequency
$|k| \gtrsim N$). Now
\begin{align*}
 &\| P_{\le N} (u^{n+1} -u^n) \|_{\infty}  \notag \\
\le&\;  \tau \left\| P_{\le N} \frac {\Delta^2 +\theta_c \Delta}  {1+\tau \Delta^2+\tau \theta_c \Delta} u^n
\right\|_{\infty}  +\tau \left\| P_{\le N}  \frac { \Delta}{1+\tau \Delta^2
+\tau \theta_c \Delta}
(\tilde f(u^n) ) \right\|_{\infty}  \notag \\
 \le &\;  O(\tau) \le \frac {\delta_1}4,
\end{align*}
if $\tau>0$ is sufficiently small.  It follows that we can guarantee $\|u^{n+1} \|_{\infty}
\le 1- {\delta_1}/2$.

We now divide the rest of the proof into several steps. The following notation will be used.
\medskip

\noindent
\textbf{Notation.} Throughout the rest of this proof, we shall denote by $C$ a generic constant
depending only $(\|u_0\|_{H^5}, \delta_0, 
\theta, \theta_c)$. The value of $C$ can change from line to line. Sometimes for a quantity $X$
we use the notation $X\lesssim 1$ to denote $X\le C$.  We denote by $C_{A_1}$ a constant whose
value depends on $A_1$. The value of $C_{A_1}$ may vary from line to line.

\subsection{Discrete energy estimate of $u^{n+1}$}

Multiplying both sides of \eqref{A1e1} by $(-\Delta)^{-1}(u^{n+1}-u^n)$ and integrating
(Taylor expand $\tilde F(u^{n+1})$ around $\tilde F(u^n)$),
we obtain
\begin{align}
&\frac 1 {\tau} \| |\nabla|^{-1}(u^{n+1}-u^n) \|_2^2
+\frac 12 \| \nabla (u^{n+1}-u^n) \|_2^2 
+\mathcal E (u^{n+1} ) - \mathcal E (u^n)  \notag \\
= & \; \frac {\theta_c}2
\| u^{n+1}-u^n\|_2^2+\frac 12 \int_{\Omega} \frac {\theta}{1-\xi_{n+1}^2} (u^{n+1}-u^n)^2 dx,
\end{align}
where $\xi_{n+1}$ is between $u^n$ and $u^{n+1}$. Since 
$\|u^{n}\|_{\infty} \le 1-\delta_1$ and $\|u^{n+1}\|_{\infty}\le
1-\frac{\delta_1}2$, we obtain $\|\xi_{n+1}\|_{\infty} \le 1- \frac {\delta_1}2$.  
Now note that 
\begin{align}
&\frac 1 {\tau} \| |\nabla|^{-1}(u^{n+1}-u^n) \|_2^2
+\frac 12 \| \nabla (u^{n+1}-u^n) \|_2^2  \notag \\
\ge & 2 \sqrt{ \frac 1 {8\tau} } \| u^{n+1} -u^n \|_2^2 
+  \frac 1 {2\tau} \| |\nabla|^{-1}(u^{n+1}-u^n) \|_2^2
+\frac 14 \| \nabla (u^{n+1}-u^n) \|_2^2.
\end{align}

Thus if $\tau>0$ is sufficiently small such that
\begin{align}
2 \sqrt{ \frac 1{8\tau} } \ge \frac {\theta_c}2 + \frac 12 \frac {\theta} { 1- (1-\frac {\delta_1} 2)^2},
\end{align}
we can guarantee the energy stability:
\begin{align}
&\frac 1 {2\tau} \| |\nabla|^{-1}(u^{n+1}-u^n) \|_2^2
+\frac 1 4 \| \nabla (u^{n+1}-u^n) \|_2^2
+\mathcal E (u^{n+1} ) - \mathcal E (u^n)  \le 0.
\end{align}
This also yields
\begin{align} \label{Au6.0e1}
\frac 1{2\tau} \sum_{j=0}^n   \| |\nabla|^{-1} (u^{j+1}-u^j) \|_2^2
+\frac 1 4 \sum_{j=0}^n \| \nabla (u^{j+1} -u^j )\|_2^2
\le \mathcal E (u^0).
\end{align}

\subsection{Preliminary estimate of $K^{n+1}$}

Denote 
\begin{align*}
K^{n+1}= -\Delta u^{n+1} -\theta_c u^{n+1} + \tilde f (u^{n+1}).
\end{align*}
Note that 
\begin{align} \label{Au6.0e5}
\frac {u^{n+1}-u^n }{\tau} = \Delta K^{n+1} +\Delta (\tilde f(u^n) -\tilde f(u^{n+1}) ).
\end{align}

\begin{lem} \label{Au1Lem0}
It holds that
\begin{align}
\|\nabla(  \tilde f(u^{n+1}) -\tilde f(u^n) )\|_2  \le C_{A_1}
\| \nabla (u^{n+1}-u^n )\|_2.
\end{align}
Thus if $\tau C_{A_1}^2 \le 1$, we have
\begin{align}
 \tau \|\nabla(  \tilde f(u^{n+1}) -\tilde f(u^n) )\|_2^2  \le \frac 12
\| \nabla (u^{n+1}-u^n )\|_2^2.
\end{align}
\end{lem}

\begin{proof} 
We write
\begin{align*}
\tilde f(u^{n+1}) -\tilde f(u^n) = \alpha_{n+1} \cdot (u^{n+1}-u^n),
\end{align*}
where  $\alpha_{n+1}= \int_0^1 \tilde f^{\prime}( u^n +\theta (u^{n+1}-u^n) ) d
\theta $. Since 
$\|u^{n}\|_{\infty} \le 1-\delta_1$, $\|u^{n+1}\|_{\infty}\le
1-\frac{\delta_1}2$ and $\|u^{n+1}\|_{H^5} \le C_{A_1}$,  we clearly have
\begin{align}
\| \nabla (\tilde f(u^{n+1}) - \tilde f(u^n) )\|_2 
& \le \| \nabla \alpha_{n+1} \|_{\infty} \| u^{n+1}-u^n \|_2
+ \| \alpha_{n+1} \|_{\infty} \|\nabla( u^{n+1}-u^n)\|_2 \notag \\
& \le C_{A_1} \|\nabla (u^{n+1} -u^n)\|_2,
\end{align}
where we have used the Poincar\'e inequality $\|\nabla (u^{n+1} -u^n)\|_2 \ge
\| u^{n+1}-u^n\|_2$. 
\end{proof}

By Lemma \ref{Au1Lem0}, \eqref{Au6.0e5} and \eqref{Au6.0e1},
it follows that for sufficiently small $\tau$, we have 
\begin{align} \label{Au6.0e3}
&\tau \sum_{j=0}^n \| \nabla K^{j+1} \|_2^2 \notag \\
\le &\;
2\sum_{j=0}^n \frac {\| |\nabla|^{-1} (u^{j+1}-u^j)\|_2^2} {\tau}
+2\sum_{j=0}^n \tau \| \nabla (\tilde f(u^{j+1}) -\tilde f(u^j) ) \|_2^2 
\le C.
\end{align}

\subsection{Long time estimate of $K^{n+1}$} \label{subs:Knlong}

Now we consider the evolution equation for $K^{n+1}$.  We have
\begin{align}
\frac {K^{n+1} -K^n} 
{\tau} &= -\Delta^2 K^{n+1} 
 -\theta_c \Delta K^{n+1} +\frac {\tilde f(u^{n+1})-\tilde f(u^{n})} {\tau}
 \notag \\
 &\qquad +\Delta^2 (\tilde f(u^{n+1}) -\tilde f(u^n) )+
 \theta_c \Delta ( \tilde f(u^{n+1}) -\tilde f(u^n) ).
\end{align}

Multiplying both sides by $-\Delta K^{n+1}$ and integrating, we obtain
\begin{align*}
&\frac {\| \nabla K^{n+1} \|_2^2 - \| \nabla K^{n} \|_2^2} {2\tau}
+\frac {\| \nabla (K^{n+1}-K^n) \|_2^2}{2\tau} \notag \\
\le&\,  - \| \Delta \nabla K^{n+1}\|_2^2 +\theta_c \| \Delta K^{n+1}\|_2^2 \notag \\
&\qquad\quad + \langle \frac {\tilde f(u^{n+1})-\tilde f(u^{n})} {\tau}, 
(-\Delta K^{n+1}) \rangle + \langle H_1, (-\Delta K^{n+1}) \rangle,
\end{align*}
where $H_1= \Delta^2 (\tilde f(u^{n+1}) -\tilde f(u^n) )+
 \theta_c \Delta ( \tilde f(u^{n+1}) -\tilde f(u^n) )$. 
 
We first deal with the term  $\langle \frac {\tilde f(u^{n+1})-\tilde f(u^{n})} {\tau}, 
(-\Delta K^{n+1}) \rangle$.  Rewrite
\begin{align*}
\frac {\tilde f(u^{n+1})-\tilde f(u^{n})} {\tau} = \alpha_{n+1}
\frac {u^{n+1} -u^{n}}{\tau}
= \alpha_{n+1} \Delta K^n +
\alpha_{n+1} \Delta (\tilde f(u^n) -\tilde f(u^{n+1}) ),
\end{align*}
where $\alpha_{n+1}= \int_0^1 \tilde f^{\prime}( u^n +\theta (u^{n+1}-u^n) ) d
\theta $.  Note that $\alpha_{n+1}\ge 0$. 
We then have
\begin{align}
 &\langle \frac {\tilde f(u^n)-\tilde f(u^{n-1})} {\tau}, 
(-\Delta K^{n+1}) \rangle  \notag \\
\le &\;\langle \alpha_{n+1} \nabla( \tilde f(u^{n}) -\tilde f(u^{n+1}) ), \Delta \nabla K^{n+1} \rangle 
+ \langle  \nabla \alpha_{n+1} \cdot \nabla( \tilde f(u^{n}) -\tilde f(u^{n+1}) ), \Delta  K^{n+1} \rangle  \notag \\
\le &\; \frac 1 8 \| \Delta \nabla K^{n+1}\|_2^2 
+2 \| \alpha_{n+1} \nabla (\tilde f(u^{n+1} -\tilde f(u^n) ) \|_2^2 \notag \\
&\qquad +C_{\epsilon} \| \nabla \alpha_{n+1} \cdot \nabla (\tilde f(u^{n+1} -\tilde f(u^n) ) \|_{\frac 43}^2 
+\epsilon \| \Delta K^{n+1}\|_4^2. 
\end{align}
By Sobolev embedding we have 
$\| \Delta K^{n+1}\|_4 \le \operatorname{const}  \| \nabla \Delta K^{n+1} \|_2$. Also
observe that
\begin{align*}
\| \nabla \alpha_{n+1} \|_4 \le C_{A_1}.
\end{align*}
Now taking  $\epsilon>0$ sufficiently small, we obtain 
\begin{align}
 &\langle \frac {\tilde f(u^{n+1})-\tilde f(u^{n})} {\tau}, 
(-\Delta K^{n+1}) \rangle  \notag \\
&\quad \le \frac 14 \| \Delta \nabla K^{n+1}\|_2^2
+C_{A_1} \| \nabla (\tilde f(u^{n+1}) -\tilde f(u^n) ) \|_2^2.
\end{align}

By Lemma \ref{Au1Lem0}, we then  have
\begin{align*}
 &\langle \frac {\tilde f(u^{n+1})-\tilde f(u^{n})} {\tau}, 
(-\Delta K^{n+1}) \rangle  \notag 
 \le \frac 14 \| \Delta \nabla K^{n+1}\|_2^2
+C_{A_1} \| \nabla (u^{n+1} -u^n ) \|_2^2.
\end{align*}

\begin{lem}
Recall $H_1= \Delta^2 (\tilde f(u^{n+1}) -\tilde f(u^n) )+
 \theta_c \Delta ( \tilde f(u^{n+1}) -\tilde f(u^n) )$.  
 Assume 
 \begin{align*}
 \|u^n\|_{H^5} \le C_1,
 \end{align*}
 where $C_1>0$ is a constant. 
 Then we have
 \begin{align} \label{2.14}
&|\langle H_1, (-\Delta K^{n+1}) \rangle |  \notag \\
\le &\;
\sqrt{\tau} \cdot C_2 \| \Delta \nabla K^{n+1} \|_2
+\theta_c \| \nabla (\tilde f(u^{n+1}) - \tilde f(u^n) ) \|_2 \| \Delta
\nabla K^{n+1} \|_2,
\end{align}
where $C_2>0$ depends on $C_1$. 
\end{lem}
\begin{proof}
We focus on the contribution of the term
$\Delta^2 (\tilde f(u^{n+1}) -\tilde f(u^n) )$. Since by assumption
$\|u^n\|_{H^5} \le C_1$, it is not difficult to obtain
$\|u^{n+1}\|_{H^5} \le \tilde C_1$ for some constant $\tilde C_1$ depending on $C_1$.
We then write $\tilde f(u^{n+1}) -\tilde f(u^n) =\alpha_{n+1} \cdot (u^{n+1}-u^n)$ as before,
and observe that
\begin{align}
u^{n+1}-u^n =
- \frac {\tau \Delta^2+\tau \theta_c \Delta} {1+\tau \Delta^2+\tau \theta_c \Delta} u^n + \frac {\tau \Delta}{1+\tau \Delta^2
+\tau \theta_c \Delta}
(\tilde f(u^n) ).
\end{align}
Clearly $\| u^{n+1}-u^n \|_{H^3} \le  \sqrt{\tau} \cdot \tilde C_2$, where $\tilde C_2$ depends
on $C_1$. It is also not difficult to check that $\| \alpha_{n+1} \|_{H^3} \le \tilde C_3$
for some $\tilde C_3$ depending on $C_1$. 
We then obtain 
$$
|\langle \Delta^2 (\tilde f(u^{n+1}) -\tilde f(u^n) ), \Delta K^{n+1} \rangle|
\le \sqrt{\tau} C_2 \|\Delta \nabla K^{n+1} \|_2.$$
The desired estimate \eqref{2.14} then easily follows.
\end{proof}

Now note
\begin{align*}
\theta_c \| \Delta K^{n+1}\|_2^2  \le \frac 14 \| \Delta \nabla K^{n+1}\|_2^2
+C \| \nabla K^{n+1} \|_2^2.
\end{align*}

Collecting all the estimates, we have
\begin{align*}
&\frac {\| \nabla K^{n+1} \|_2^2 - \| \nabla K^{n} \|_2^2} {2\tau} \notag \\
\le &\;-\frac 14 \| \Delta \nabla K^{n+1}\|_2^2+\sqrt{\tau} \cdot C_2 \| \Delta \nabla K^{n+1} \|_2
+ C \| \nabla K^{n+1}\|_2^2 +  C_{A_1} \| \nabla (u^{n+1}-u^n)\|_2^2.
\end{align*}

Now take $\tau$ sufficiently small such that $\sqrt{\tau} C_2 \le \frac 14$, $\tau C_{A_1} \le \frac 14$.

We discuss two cases. 

\vspace{0.25cm}

Case 1: $\|\Delta \nabla K^{n+1} \|_2 \le 10$.  In this
case we call such $n$ a good point.  In this case, no work is needed since by
Poincar\'e inequality we have $\|\nabla K^{n+1}\|_2 \le 10$.

\medskip

Case 2: $\| \Delta \nabla K^{n+1}\|_2>10$.  In this case note that
$\| \Delta \nabla K^{n+1} \|_2^2 \ge 10 \| \Delta \nabla K^{n+1}\|$.
Thus
 we obtain
\begin{align}
&\frac {\| \nabla K^{n+1} \|_2^2 - \| \nabla K^{n} \|_2^2} {2\tau} \notag \\
\le&\, -\frac 1{400} \| \Delta \nabla K^{n+1}\|_2^2
+ C \| \nabla K^{n+1}\|_2^2 +  \frac 1{4\tau} \| \nabla (u^{n+1}-u^n)\|_2^2.
\label{Au6.1e50}
\end{align}
Recall that we have shown (see \eqref{Au6.0e1} and \eqref{Au6.0e3})
\begin{align}
&\tau \sum_{j=0}^n \| \nabla K^{j+1} \|_2^2 \le C; \quad
\sum_{j=0}^n \| \nabla(u^{j+1}-u^j) \|_2^2 \le C.
\end{align}
Now using \eqref{Au6.1e50} and summing backwards in $n$ until one meets a good $n$ or $n=0$,  we then obtain
\begin{align}
\sup_{0\le j \le n} \| \nabla K^{j+1} \|_2 \le C <\infty.
\end{align}

\subsection{Control of $\| g^{n+1} -\overline{g^{n+1}} \|_2$}

We shall use $\| \nabla K^{n+1} \|_2 \le C $ which gives $\| K^{n+1}-\overline{K^{n+1}}\|_2
\le C$. Write
\begin{align*}
 K^{n+1} - \overline{ K^{n+1}}
= -\Delta u^{n+1} - \theta_c (u^{n+1}-\overline{u^{n+1}})
+ \theta (g^{n+1}-\overline{g^{n+1}}).
\end{align*}
Multiplying both sides by $g^{n+1} -\overline{g^{n+1}}$, 
integrating (note the simple inequality
 $ |\langle u^{n+1}, g^{n+1}-\overline{g^{n+1}} \rangle|
\le \| g^{n+1}-\overline{g^{n+1}}\|_2$) and using the Cauchy-Schwartz inequality,
 we obtain
\begin{align}
\| g^{n+1} - \overline{g^{n+1}} \|_2  \le C.
\end{align}
In the above derivation we used the fact that
\begin{align}
\int_{\mathbb T^2} (-\Delta u^{n+1}) (g^{n+1}
-\overline{g^{n+1}} ) dx = \int_{\mathbb T^2} |\nabla u^{n+1}|^2 \frac {\theta}
{1-(u^{n+1})^2} dx \ge 0.
\end{align}

\subsection{Control of $\overline{g^{n+1} }$ and $\overline{K^{n+1}}$ }

\begin{lem} \label{Au1Lem1}
Assume $\| g -\bar g \|_2 \lesssim 1$. 
Let $u=\tanh(g)$ and $|\overline{u}|<1$. Then
\begin{align}
|\overline{g} | \lesssim  (1-|\bar u|)^{-\frac 12}.
\end{align}
\end{lem}

\begin{proof}
If $|\bar g |\le 10$ we are done. Now we assume $\bar g = M \ge 10$. Since
$\| g-\bar g\|_2 \lesssim 1$, we obtain 
\begin{align*}
\operatorname{Leb}\{x\in \Omega:\, g(x)\le M/2 \} \lesssim M^{-2}.
\end{align*}

Now
\begin{align}
\bar u \operatorname{Leb}(\Omega) &= \int_{g(x)\ge \frac M2} dx + \int_{g(x)\ge \frac M2} (u(x)-1) dx +
\int_{g(x)<\frac M2} u(x) dx \notag \\
&=\operatorname{Leb}(\Omega) + \int_{g(x)\ge \frac M2} (u(x)-1) dx + \int_{g(x)<\frac M2} (u(x)-1) dx \notag \\
& = \operatorname{Leb}(\Omega) + O(e^{-\frac M4}) +O(\frac 1{M^2}).
\end{align}
Clearly then 
\begin{align}
(1-\bar u) \operatorname{Leb}(\Omega) \lesssim M^{-2}.
\end{align}
Thus the desired inequality follows. Note that if $\bar g \le -10$ we need to work with $-u$
and hence the bound of $\bar g$ depends on $(1-|\bar u|)^{-\frac 12}$. 
\end{proof}

Since $\overline{u^{n+1}} = \overline{u^0} $ is preserved in time and $|\overline{u^0}|<1$,
Lemma \ref{Au1Lem1} implies that $|\overline{g^{n+1} }| \lesssim 1$. 

\medskip

For the control of $\overline{K^{n+1}} $, recall that
\begin{align} \label{Au2e1}
K^{n+1}= -\Delta u^{n+1} -\theta_c u^{n+1} + \theta g^{n+1}. 
\end{align}
 Clearly then
\begin{align}
|\overline{K^{n+1}}| \le \theta_c |\overline{u^{n+1}}| + \theta|\overline{g^{n+1}}|
\lesssim 1.
\end{align}

\subsection{Control of $\|u\|_{H^3}$, $\|g \|_{H^3}$, $\| \tilde f(u) \|_{H^3}$,
$\| \frac 1 {1-u^2} \|_{\infty}$, $\| \partial( \frac {1}{1-u^2} ) \|_{\infty}$
and $\| \partial^2 ( \frac 1 {1-u^2} ) \|_{\infty-}$}

 Here $g=g^{j}$ and $u=u^{j}$, 
$1\le j\le n+1$. The obtained estimates will be uniform in $j$. See the subsection ``Notation"
for the definition of $\|f \|_{\infty-}$.

We shall explain the argument for $j=n+1$. It is clear from the argument below that the 
estimates will be uniform in $j$.

Since $ \| \nabla K^{n+1} \|_2 \lesssim 1$ and we have the control of $\overline{K^{n+1}}$, it 
follows that  
\begin{align*}
\| K^{n+1} \|_p \lesssim \sqrt p, \quad \forall\, 2\le p<\infty.
\end{align*}
By using \eqref{Au2e1} (multiply both sides by $|g^{n+1}|^{p-2}g^{n+1}$ and integrate by parts),  we then get
\begin{align*}
\| g^{n+1} \|_{p} \lesssim \sqrt p.
\end{align*}
This implies for any $C>0$,
\begin{align}
\| e^{C|g^{n+1}|} \|_{\infty-} \lesssim 1.
\end{align}
Since $\frac 1 {1-(u^{n+1})^2} \lesssim e^{2|g^{n+1}|}$, we also get $\| \frac 1 {1-(u^{n+1})^2} \|_{\infty-} \lesssim 1$.

By using \eqref{Au2e1}, we also get $\| \Delta u^{n+1} \|_{\infty-} \lesssim 1$. This easily implies
\begin{align}
\left\| \partial^2 ( \frac 1 {1-(u^{n+1})^2} ) \right\|_{\infty-} \lesssim 1.
\end{align}
By Sobolev embedding, we  obtain $\|\frac 1 {1-(u^{n+1})^2} \|_{\infty} \lesssim 1$. 
Since
\begin{align*}
|\ln(\frac {1+x}{1-x}) | \lesssim \frac 1 {1-x^2}, \qquad\text{for $|x|<1$,}
\end{align*}we also obtain
$\|g^{n+1}\|_{\infty} \lesssim 1$.  Since \[\| \nabla g^{n+1}\|_2
\lesssim \| \frac 1 {1-(u^{n+1})^2} \nabla u^{n+1} \|_2 \lesssim 1,\] by using
\eqref{Au2e1}, we obtain $\|u^{n+1} \|_{H^3} \lesssim 1$.  This further implies
that $$\| \nabla g^{n+1} \|_{H^2} \lesssim \| \frac 1 {1-(u^{n+1})^2}
\nabla u^{n+1} \|_{H^2} \lesssim 1.$$ It is also clear that
$\| \tilde f(u^{n+1} )\|_{H^3} \lesssim 1$.

\subsection{Control of $\|u^{n+1} \|_{H^5}$}

Here we shall exploit the discrete smoothing effect. Denote
\begin{align}
u^{n+1} &=  \frac 1 {1+\tau \Delta^2+\tau \theta_c \Delta} u^n + \frac {\tau \Delta}{1+\tau \Delta^2
+\tau \theta_c \Delta}
(\tilde f(u^n) ) \notag \\
&= :T_0 u^n +T_1 f^n. 
\end{align}
Iterating the above gives
\begin{align*}
u^{n+1} = T_0^{J+1} u^{n-J} + \sum_{j=1}^J T_0^j T_1 f^{n-j} +T_1 f^n.
\end{align*}

In the estimate below, we shall use the uniform estimate:
\begin{align}
\|u^0\|_{H^5}+ \sup_{0\le j\le n} (\| u^j\|_{H^3} +\|f^j \|_{H^3} )\le B_1<\infty.
\end{align}

\subsubsection{Discrete smoothing estimates}
We first prove two auxiliary lemmas needed for the higher order estimates later. 
In a slightly more general setup, we assume for some $s\ge 0$, 
\begin{align*}
\sup_{0\le j\le n} \|\widehat{h^j}(k) |k|^s \|_{l_k^{\infty}} \le \alpha_1<\infty. 
\end{align*}

\noindent
Define 
\begin{align*}
v =  \sum_{j=1}^n T_0^j T_1 h^j.
\end{align*}

\begin{lem} \label{Ap2}
We have
\begin{align}
\| \widehat{v}(k) |k|^{s+2} \|_{l_k^{\infty}} \le \alpha_1.
\end{align}
Consequently $\|v \|_{H^{s+0.9}(\mathbb T^2)} \le \alpha_2<\infty$, where $\alpha_2>0$
depends only on $\alpha_1$. 
\end{lem}

\begin{proof}
Observe that
\begin{align*}
|\widehat{T_0}(k) \widehat{T^j_1}(k)|
\le \left(\frac 1 {1+\tau |k|^4}\right)^j \cdot \frac {\tau |k|^2}
{1+\tau |k|^4}. 
\end{align*}
Then for each $k\ne 0$, we have 
\begin{align}
|\widehat v (k)| \le \alpha_1 |k|^{-s}  \sum_{j=1}^n
\left(\frac 1 {1+\tau |k|^4}\right)^j \cdot \frac {\tau |k|^2} 
{1+\tau |k|^4} \le \alpha_1 |k|^{-s-2}. 
\end{align}
Thus $\|v \|_{H^{s+0.9}(\mathbb T^2)}$ is bounded.
\end{proof}

\begin{lem} \label{Ap3}
Assume $0<\tau\le 1$ and $4\le J\tau <5$. Then
\begin{align}
\| T_0^J g \|_{H^{16}(\mathbb T^2)} \lesssim \| g\|_{2}. 
\end{align}

\end{lem}
\begin{proof}
Observe that $J\ge 4$ and 
\begin{align}
|\widehat{T_0^J}(k) | \le (1+|k|^4 \frac {J\tau}J )^{-J} 
\le (1+ |k|^4 \frac 4 J )^{-J} \le (1+|k|^4)^{-4}.
\end{align}
Here we used the simple inequality that $h(x)= (1+\frac 1 x )^x$ is monotonically increasing
in $x>0$. Clearly the desired inequality then follows.
\end{proof}

\subsubsection{Higher order estimates}

Now we discuss two cases.
\vspace{0.25cm}

Case $1$: $n\tau \le 20$.  In this case we take $J=n$. By Lemma \ref{Ap2}, we have
\begin{align*}
\| |k|^5 \widehat{ u^{n+1}} (k) \|_{l_k^{\infty} } \le C. 
\end{align*}
By a similar estimate we also obtain
\begin{align*}
\sup_{0\le j \le n+1} \| |k|^5 \widehat{u^{j}}(k) \|_{l_k^{\infty}} \le C.
\end{align*}
It follows that 
\begin{align*}
\sup_{0\le j\le n+1} \| |k|^5 \widehat{ \tilde 
 (u^j)} (k) \|_{l_k^{\infty}} \le C. 
\end{align*}
By using Lemma \ref{Ap2} again, we obtain $\|u^{n+1} \|_{H^5} \le C$. 

\vspace{0.25cm}

Case $2$: $n\tau >20$.  Since $0<\tau\le 1$, we can choose an integer 
$J\ge 4$ such that $4\le J\tau <5$. By using Lemma \ref{Ap3}, we have
\begin{align*}
\| T_0^{J+1} u^{n-J} \|_{H^{10}} \le C.
\end{align*}
The inhomogeneous term containing $f^{n-j}$ can be handled in the same
way as in Case $1$.  Thus in this case we obtain
$\| u^{n+1} \|_{H^5} \le C$. This then completes the whole induction
step.  Theorem \ref{t1} is now proved. 

\section{Proof of Corollary \ref{c1}}
In this section we give the proof of Corollary \ref{c1}. Consider first $s=5$. 
For each small $\tau>0$ we denote $u^{n,\tau}=u^n$ as the numerical solution obtained with the
help of Theorem \ref{t1}.  Define $v^{(\tau)} \in C_t^0 H_x^5$ such that
\begin{align}
v^{(\tau)} (t) :=
\begin{cases}
u^{n,\tau}, \quad \text{if $t=n\tau$, $n\ge 0$}, \\
(n+1 - \frac t{\tau})u^{n,\tau} + (\frac t{\tau}-n)u^{n+1,\tau}, \quad \text{if $n\tau \le t<(n+1)\tau$, $n\ge 0$}.
\end{cases}
\end{align}
In yet other words, $v^{(\tau)}$ is the piece-wise linear interpolation of $(u^{n,\tau})_{n\ge 0}$.
Observe that for each $t \in (n\tau, (n+1)\tau)$, we have
\begin{align} \label{Au5e12a}
\partial_t v^{(\tau)} =
\frac {u^{n+1,\tau }-u^{n,\tau}}
{\tau}= -\nu \Delta^2 u^{n+1,\tau} -\theta_c \Delta u^{n+1,\tau} +
\Delta (\tilde f(u^{n,\tau} ) ).
\end{align}

By Theorem \ref{t1}, we have 
\begin{align*}
\sup_{0<\tau\le \tau_0}( \| v^{(\tau)} \|_{C_t^0 H^5_x}+
\| \partial_t v^{(\tau)} \|_{L_t^{\infty} H_x^1}  )\lesssim 1.
\end{align*}
In the above, the norms are evaluated on the space-time slab $[0,\infty)\times \mathbb T^2$, and
we have used the fact that the quotients $\frac{u^{n+1,\tau} -u^{n,\tau}} {\tau}$ are uniformly bounded in $L_t^{\infty} H_x^1$.

To proceed further, we need the following variant of
the usual Aubin-Lions compactness lemma. 

\begin{lem} \label{Au5lem1}
Suppose $(v_n)_{n\ge 1}$ is a sequence of functions defined on $(t,x) \in [0,\infty)
\times \mathbb T^2$ such that $v_n \in C_t^0 H^5_x$, $\partial_t v_n \in L_t^{\infty} H_x^1$
and 
\begin{align}
\sup_{n\ge1 } \left( \| v_n \|_{C_t^0 H^5_x}+
\| \partial_t v_n \|_{L_t^{\infty} H_x^1}  \right)\lesssim 1.
\end{align}

Then there exists
$v_{*} \in C_t^0 H^5_x$ with $\partial_t v_{*} \in L_t^{\infty} H_x^1$, and a subsequence
$n_j$, such that for any given $T>0$, and any $1\le s^{\prime}<5$, 
\begin{align*}
\|v_{n_j} -v_*\|_{C_t^0 H^{s^{\prime}}_x([0,T]\times \mathbb T^2)}  \to 0, 
\qquad \text{as $j\to \infty$.}
\end{align*}

\end{lem}
\begin{proof}[Proof of Lemma \ref{Au5lem1}]
Without loss of generality we assume $T=1$. From the argument below together with a further diagonal argument one can easily cover the general case $T>0$.

First fix any $k\in \mathbb Z^2$ and consider $\widehat {v_n}(t,k)$ on the time interval $[0,1]$. 
By using Arzel\`a-Ascoli and using a diagonal argument one can extract a subsequence which we denote as $\widehat{v_{n_j}}(t,k)$
such that it converges to a continuous (in $t$) function $\widehat{v_*}(t,k)$ on $[0,1]$ for
any fixed $k$. Furthermore, since $|\widehat{v_n}(t,k)|\lesssim (1+|k|)^{-5}$, one can obtain
\begin{align*}
\|\widehat{v_{n_j}}(t,k) - \widehat{v_*}(t,k)\|_{L_t^{\infty} l_k^2} \to 0,
\qquad \text{as $j\to \infty$.}
\end{align*}
 By using
interpolation one can  obtain the strong convergence in $C_t^0 H^{s^{\prime}}$ for any
$s^{\prime}<5$. 
\end{proof}

By using Lemma \ref{Au5lem1}, we obtain that along some sequence $\tau_j \to 0$, 
there exits $u \in C_t^0 H^5$ with $\partial_t u \in L_t^{\infty} H^4$ , such that
as $j\to \infty$, 
\begin{align}
\| v^{(\tau_j)} - u \|_{C_t^0 H^{4.5} ([0, T]\times \mathbb T^2)} \to 0,
\qquad \forall\, T>0. 
\end{align}
Now by \eqref{Au5e12a}, it is not difficult to check that for any test function $\psi \in
C_c^{\infty}((0,T)\times \mathbb T^2)$, we have 
\begin{align}
&\left|\int_{(0,T)\times \mathbb T^2}  \psi(t,x) \Bigl(  \partial_t v^{(\tau)}  +\nu \Delta^2 v^{(\tau)} 
+\theta \Delta v^{(\tau)} -\Delta (\tilde f(v^{(\tau) } ) \Bigr) dx dt\right|  \notag\\
\le&\;  O(\tau) \to 0, \qquad \text{as $\tau\to 0$}. 
\end{align}

This together with the regularity of $u$ implies that $u$ is the desired solution. Note
that the strict phase separation and uniform Sobolev regularity of $u$ on the time
interval $[0,\infty)$ follows by taking the limit. Thanks to strict phase separation, it is routine
to check that our constructed solution is unique.  
We note that the general case $s>5$ can be obtained by a simple bootstrapping argument.
We omit further details here and leave them to interested readers.

\section{Proof of Theorem \ref{t2}}
In this section we carry out the error estimate in $L^2$.
\subsection{Auxiliary $L^2$ error estimate for near solutions}
Consider
\begin{align} \label{ech_e4.1}
\begin{cases}
\displaystyle\frac{v^{n+1}-v^n}{\tau} = - \nu \Delta^2 v^{n+1} - \theta_c \Delta v^{n+1}+\Delta  \tilde f(v^n) +\Delta \tilde
G_n^1, \quad n\ge 0,\\
\displaystyle\frac{\tilde v^{n+1} -\tilde v^n} {\tau} =-\nu \Delta^2 \tilde
v^{n+1} -\theta_c \Delta \tilde v^{n+1} + \Delta  \tilde f(\tilde
v^n)+
\Delta \tilde G^n_2, \qquad n\ge 0,\\
v^0=v_0, \quad \tilde v^0=\tilde v_0,
\end{cases}
\end{align}
where $v_0$ and $\tilde v_0$ have the same mean and 
$\tilde f(z)=\frac{\theta}2 \ln(\frac{1+z}{1-z})$.  For simplicity we shall make
a slightly stronger assumption
$$0<\tau\le \frac {\nu}{4\theta_c^2}$$
so that the operator $(1+\tau \nu \Delta^2 + \theta_c \tau \Delta )$ is  invertible
and consequently $v^{n+1}$, $\tilde v^{n+1}$ are well-defined for all $n\ge 0$. 
Denote $\tilde G^n=\tilde G^n_1-\tilde G^n_2$.

\begin{prop} \label{propy2}
Assume $0<\tau\le \frac{\nu}{4\theta_c^2}$. 
For solutions of \eqref{ech_e4.1}, assume for some $\delta_1\in (0,1)$, 
\begin{align}
&\sup_{n\ge 0} \|  v^n \|_{\infty} \le 1-\delta_1, \quad \sup_{n\ge 0}
\|\tilde v^n\|_{\infty}\le 1-\delta_1.
\end{align}
Then for any $m\ge 1$, we have
\begin{align}
& \| v^m -\tilde v^m \|_2^2 
\le \; \exp\Bigl({m\tau\cdot
\frac{C_1}{\nu}}\Bigr)\cdot \Bigl( \|  v_0-\tilde
v_0\|_2^2 +
\frac{2\tau}{\nu} \sum_{n=0}^{m-1} \| \tilde G^n \|_2^2
\Bigr),
\end{align}
where $C_1>0$ is a constant depending only on $(\delta_1,\theta,\theta_c)$.
\end{prop}

\subsubsection*{Proof of Proposition \ref{propy2}}
Denote $e^n= v^n - \tilde v^n$. Then
\begin{align}
\frac{e^{n+1}-e^n}{\tau} = -\nu \Delta^2 e^{n+1}
-\theta_c \Delta e^{n+1} +\Delta (\tilde f(v^n)-\tilde f(\tilde v^n)) +\Delta \tilde
G^n.
\end{align}
Taking the $L^2$ inner product with $ e^{n+1}$ on both sides, we get
\begin{align}
 & \frac 1 {2\tau} ( \| e^{n+1} \|_2^2 -
 \|  e^n \|_2^2 + \| e^{n+1}-e^n\|_2^2 ) + \nu \|\Delta e^{n+1} \|_2^2 
 +\theta_c \langle e^{n+1}, \Delta e^{n+1} \rangle \notag \\
 =& \; \langle \tilde G^n, \Delta e^{n+1} \rangle 
  + \langle \tilde f(v^n)-\tilde f(\tilde v^n), \Delta  e^{n+1} \rangle. \label{yy30}
 \end{align}
By the Cauchy-Schwartz inequality, we have
\begin{align}
|(\tilde G^n, \Delta e^{n+1})| &\le \frac{ \| \tilde G^n\|_2^2} {\nu} + \frac{\nu} 4 \| \Delta e^{n+1} \|_2^2, \notag\\
|\theta_c \langle e^{n+1}, \Delta e^{n+1} \rangle|
&\le \frac{\theta_c^2}{\nu} \| e^{n+1}\|_2^2 + \frac{\nu} 4
\| \Delta e^{n+1}\|_2^2 \notag \\
&\le \frac{2\theta_c^2}{\nu} \| e^{n+1}-e^n\|_2^2 + 
\frac{2\theta_c^2}{\nu} \| e^n\|_2^2+ \frac{\nu} 4
\| \Delta e^{n+1}\|_2^2.
\end{align}
Since $\max\{\|v^n\|_{\infty},\|\tilde v^n \|_{\infty} \} \le 1-\delta_1$,  we have
\begin{align}
&\|\tilde f(v^n) -\tilde f(\tilde v^n) \|_2 \le C \cdot \|e^n\|_2, \notag\\
&|\langle \tilde f(v^n)-\tilde f(\tilde v^n), \Delta  e^{n+1} \rangle|\le
\frac {\nu}4 \|\Delta e^{n+1}\|_2^2 + \frac{C^2}{\nu} \|e^n\|_2^2,
\end{align}
where $C>0$ depends only on $\delta_1$ and $\theta$. 

Collecting the estimates, we get
\begin{align}
&\frac{\|e^{n+1}\|_2^2 - \| e^n\|_2^2}{2\tau} + (\frac 1{2\tau}-\frac{2\theta_c^2}
{\nu})\| e^{n+1}-e^n\|_2^2  \notag \\
\le&\; \frac {1}\nu \|\tilde G^n\|_2^2 +  \frac{C^2+2\theta_c^2}{\nu} \|e^n\|_2^2.
\end{align}

\noindent
Define
\begin{align*}
y_n =\|e^n\|_2^2, \quad \tilde \alpha= \frac{2C^2+4\theta_c^2} {\nu},
\quad \tilde \beta_n= \frac{2}{\nu}\|\tilde G^n\|_2^2.
\end{align*}

\noindent
Then obviously
\begin{align*}
\frac{y_{n+1}-y_n}{\tau} \le \tilde \alpha y_n +\tilde \beta_n.
\end{align*}

The desired result  follows from the standard Gronwall inequality. \qed

\vspace{0.25cm}
\noindent
Next we state and prove two lemmas needed for the proof of Theorem \ref{t2}.

\begin{lem}[Discretization of the PDE solution] \label{Ap5}
Let $t_n=n\tau$, $n\ge 0$. Let $u$ be the exact PDE solution to \eqref{1}.
Denote $\tilde f(z)= \frac {\theta} 2 \ln (\frac {1+z}{1-z} )$. 
 We have
\begin{align}
\frac{u(t_{n+1})-u(t_n)}{\tau}
= -\nu \Delta^2 u(t_{n+1}) -\theta_c \Delta u(t_{n+1}) + \Delta ( \tilde f(u(t_n)) )
+\Delta \tilde G^n_1,
\end{align}
where 
\begin{align*}
\tilde G^n_1 &= \frac{\nu}{\tau} \int_{t_n}^{t_{n+1}}
\partial_t \Delta u \cdot (t-t_n) dt +\frac{\theta_c}{\tau}
\int_{t_n}^{t_{n+1}} \partial_t u \cdot (t-t_n) dt \notag \\
&\quad+ 
 \frac 1 {\tau} \int_{t_n}^{t_{n+1}}
\partial_t (\tilde f(u)) \cdot(t_{n+1}-t) dt.
\end{align*}
Similarly for a slightly different discretization, we have
\begin{align} \label{Ap5_e2}
\frac{u(t_{n+1})-u(t_n)}{\tau}
= -\nu \Delta^2 u(t_{n+1}) + \Delta ( f(u(t_n)) )
+\Delta \tilde G^n_2,
\end{align}
where 
\begin{align*}
\tilde G^n_2 &= \frac{\nu}{\tau} \int_{t_n}^{t_{n+1}}
\partial_t \Delta u \cdot (t-t_n) dt + 
 \frac 1 {\tau} \int_{t_n}^{t_{n+1}}
\partial_t ( f(u)) \cdot(t_{n+1}-t) dt.
\end{align*}

\end{lem}
\begin{proof}
Integrating the PDE for $u$ on the time interval $[t_n,t_{n+1}]$, we obtain
\begin{align*}
\frac{u(t_{n+1})-u(t_n)} {\tau}
=\int_{t_n}^{t_{n+1}} \Bigl( -\nu \Delta^2 u(t) + \Delta( f(u(t) ) ) \Bigr) ds.
\end{align*}

Note that for a one-variable function $h=h(t)$, we have the formula
\begin{align} \notag
\frac 1 {\tau} \int_{t_n}^{t_{n+1}} h(t) dt & = h(t_n) +
\frac 1 {\tau} \int_{t_n}^{t_{n+1}} h^{\prime}(t) \cdot(t_{n+1}-t) dt, \\
\frac 1 {\tau} \int_{t_n}^{t_{n+1}} h(t) dt & = h(t_{n+1} ) + \frac 1 {\tau} \int_{t_n}^{t_{n+1}} h^{\prime}(t) \cdot(t_{n}-t) dt. \notag
\end{align}

By using the above formula, we have
\begin{align*}
& \int_{t_n}^{t_{n+1}}  (-\nu \Delta^2 u(t) ) dt=
- \nu \Delta^2 u(t_{n+1}) + \frac 1 {\tau}
\int_{t_n}^{t_{n+1}} \nu \partial_t \Delta^2 u(t) \cdot (t-t_n) dt;\\
& \int_{t_n}^{t_{n+1}}  (-\theta_c \Delta u(t) ) dt=
- \theta_c \Delta u(t_{n+1}) + \frac 1 {\tau}
\int_{t_n}^{t_{n+1}} \theta_c \partial_t \Delta u(t) \cdot (t-t_n) dt; \\
&\int_{t_n}^{t_{n+1}} \Delta(\tilde f(u(t) ) )dt = \Delta(\tilde f(u(t_n)) ) +
\frac 1 {\tau} \int_{t_n}^{t_{n+1}} \partial_t \Delta (\tilde f(u(t) ) ) \cdot (t_{n+1}-t) dt. 
\end{align*}
Thus 
\begin{align*}
\frac{u(t_{n+1})-u(t_n)}{\tau}
= -\nu \Delta^2 u(t_{n+1}) -\theta_c \Delta u(t_{n+1}) + \Delta ( \tilde f(u(t_n)) )
+\Delta \tilde G^n,
\end{align*}
where 
\begin{align*}
\tilde G^n &= \frac{\nu}{\tau} \int_{t_n}^{t_{n+1}}
\partial_t \Delta u \cdot (t-t_n) dt +\frac{\theta_c}{\tau}
\int_{t_n}^{t_{n+1}} \partial_t u \cdot (t-t_n) dt \notag \\
&\quad+ 
 \frac 1 {\tau} \int_{t_n}^{t_{n+1}}
\partial_t (\tilde f(u)) \cdot(t_{n+1}-t) dt.
\end{align*}
The derivation of \eqref{Ap5_e2} is similar. We omit details.
\end{proof}

\begin{lem} \label{Ap6}
Let $u$ be the PDE solution constructed in Corollary \ref{c1}. Then we have
\begin{align}
\int_0^{\infty} (\| \partial_t u \|_2^2  + \| \partial_t \Delta u\|_2^2 ) dt \lesssim 1.
\end{align}

\end{lem}

\begin{proof}
The $L^2$-in time integrability of $\partial_t u$ comes from the energy identity, i.e.:
\begin{align*}
\frac d{dt} \mathcal E(u(t) ) = -\| \partial_t u \|_{L_x^2}^2 \quad
\Longrightarrow\; \mathcal E (u(T)) + \int_0^T \| \partial_t u\|_2^2 dt =\mathcal E(u_0),
\qquad \forall\, T>0.
\end{align*}
Sending $T\to \infty$ then yields $ \int_0^{\infty} \| \partial_t u \|_2^2 dt <\infty$. 

Next to obtain $L^2$-in time integrability of $\partial_t \Delta u$, we recall 
\begin{align}
\partial_t u = \Delta \mu = \Delta ( - \nu \Delta u -\theta_c u + \tilde f(u)  ).
\end{align}
Clearly 
\begin{align} \label{Ap6.2}
\partial_t \mu = - \nu \Delta^2 \mu - \theta_c \Delta \mu  +\tilde f^{\prime}(u ) \Delta \mu.
\end{align}
Thanks to strict phase separation, we have $\tilde f^{\prime}(u)\ge 0$. Multiplying both sides
of \eqref{Ap6.2} by $-\Delta \mu$ and integrating by parts, we obtain
\begin{align}
\frac 12 \partial_t ( \| \nabla \mu\|_2^2) \le - \nu \| \Delta \nabla \mu\|_2^2 
+ \theta_c \| \Delta \mu\|_2^2.
\end{align}
Integrating in time and using the fact that $\Delta \mu = \partial_t u \in L_t^2 L_x^2((0,\infty)
\times \mathbb T^2)$, we obtain
\begin{align*}
\int_0^{\infty} \| \Delta \nabla \mu \|_2^2 dt <\infty.
\end{align*}
Multiplying both sides of \eqref{Ap6.2} by $\Delta^2 \mu$ and integrating by parts,  we then
obtain
\begin{align*}
\int_0^{\infty} \| \Delta^2 \mu \|_2^2 dt <\infty.
\end{align*}
Here it should be noted that in deriving the above, we used the finiteness of $\| \Delta \mu\|_2$
which is clearly bounded since $u\in H^5$ and has strict phase separation.  Since 
$\partial_t \Delta u = \Delta^2 \mu$, we then obtain $\int_0^{\infty} \| \partial_t \Delta u\|_2^2 dt
<\infty$.
\end{proof}

\subsubsection*{Proof of Theorem \ref{t2}}

We need to consider
\begin{align}
\begin{cases}
\displaystyle\frac{ u^{n+1}- u^n} {\tau}
= - \nu \Delta^2  u^{n+1} -\theta_c \Delta u^{n+1}
+\Delta( \tilde f( u^n)), \\
\partial_t u = -\nu \Delta^2 u + \Delta( f(u)), \\
\tilde u^0 =u_0, \quad u(0)=u_0.
\end{cases}
\end{align}

\noindent
We first rewrite the PDE solution $u$ in the discretized form. By Lemma \ref{Ap5}, we have
\begin{align}
\frac{u(t_{n+1})-u(t_n)}{\tau}
= -\nu \Delta^2 u(t_{n+1}) -\theta_c \Delta u(t_{n+1}) + \Delta ( \tilde f(u(t_n)) )
+\Delta \tilde G^n,
\end{align}
where 
\begin{align*}
\tilde G^n &= \frac{\nu}{\tau} \int_{t_n}^{t_{n+1}}
\partial_t \Delta u \cdot (t-t_n) dt +\frac{\theta_c}{\tau}
\int_{t_n}^{t_{n+1}} \partial_t u \cdot (t-t_n) dt \notag \\
&\quad+ 
 \frac 1 {\tau} \int_{t_n}^{t_{n+1}}
\partial_t (\tilde f(u)) \cdot(t_{n+1}-t) dt.
\end{align*}
By using strict phase separation and uniform Sobolev regularity of $u$, we have
\begin{align}
\| \tilde G^n\|_2
&\le \nu \int_{t_n}^{t_{n+1}}
\| \partial_t \Delta u\|_2 dt
+ {\theta_c} \int_{t_n}^{t_{n+1}} \| \partial_t u \|_2 dt
+ \int_{t_n}^{t_{n+1}}
\| \partial_t u \|_2 dt  \cdot \sup_{|z|\le 1-\delta_1} |\tilde f^{\prime}(z) | \notag \\
&\lesssim \;  \left(\int_{t_n}^{t_{n+1}} (\| \partial_t \Delta u\|_2^2
+\| \partial_t u \|_2^2 )dt \right)^{\frac 12} \cdot \sqrt{\tau}. 
\end{align}
By Lemma \ref{Ap6}, we obtain
\begin{align}
\sum_{n=0}^{m-1} \| \tilde G^n\|_2^2
& \lesssim   \tau \int_0^{t_m} (\| \partial_t \Delta u\|_2^2 +\|\partial_t u\|_2^2) dt \lesssim
\tau.
\end{align}
Thus
by Proposition \ref{propy2}, we  get
\begin{align}
\|  u^m -u(t_m) \|_2^2 \lesssim e^{Ct_m} \tau^2.
\end{align}
Thus we obtain \eqref{1.21}.  \qed

\section{Concluding remarks}
In this paper we studied the Cahn-Hilliard equation with singular logarithmic potentials on the
two-dimensional periodic torus. We analyzed a first order in time, semi-implicit numerical discretization scheme  which treats the linear fourth-order dissipation term implicitly and the nonlinear
 term explicitly.  Prior state of the art literature are concerned with implicit or partially implicit
 methods for which phase separation and energy stability are established under nearly optimal
 conditions. For semi-implicit type methods, these issues were long standing open problems.
 In this work  we developed a new theoretical framework and  proved strict phase
 separation and energy stability for all time under mild constraints on the time step and initial data. 
We also carried out a rigorous error analysis which is done for the first time for semi-implicit
methods on Cahn-Hilliard equations with singular potentials.  It is expected our theoretical framework 
can be refined and  generalized to cover many other similar problems. 
Research is now underway
to investigate several directions including the stability and error analysis of higher-order methods, 
general thin-film type problems with singular potentials, various 
time-splitting methods,  and adaptive time-stepping methods.

\appendix

\frenchspacing
\bibliographystyle{plain}

\end{document}